\documentclass[british,oneside]{amsart}
\usepackage[T1]{fontenc}
\usepackage{babel}

\usepackage[svgnames]{xcolor}
\usepackage[pdfusetitle,colorlinks,citecolor=ForestGreen]{hyperref}
\usepackage{mathtools}
\usepackage{cleveref}

\usepackage{lmodern}
\usepackage{microtype}

\DeclarePairedDelimiter{\fspc}{(\!(}{)\!)}

\DeclarePairedDelimiter{\wspc}{\langle\!\langle}{\rangle\!\rangle}

\newcommand{\Jb}{\mathbb{J}}
\newcommand{\Mf}{\mathfrak{M}}
\newcommand{\mf}{\mathfrak{m}}
\newcommand{\Nf}{\mathfrak{N}}
\newcommand{\Of}{\mathfrak{O}}
\newcommand{\Nb}{\mathbb{N}}
\newcommand{\Rb}{\mathbb{R}}
\newcommand{\Tb}{\mathbb{T}}
\newcommand{\Ub}{\mathbb{U}}

\newcommand{\ER}{\mathrm{ER}}
\newcommand{\LE}{\mathrm{LE}}
\newcommand{\No}{\mathbf{No}}
\newcommand{\OM}{\Rb\wspc{T}}
\newcommand{\On}{\mathbf{On}}

\newtheorem{mainthm}{Theorem}

\newtheorem{maincor}[mainthm]{Corollary}
\crefname{mainthm}{Theorem}{Theorems}
\crefname{maincor}{Corollary}{Corollaries}

\newtheorem{thm}{Theorem}[section]
\newtheorem{prop}[thm]{Proposition}
\newtheorem{lem}[thm]{Lemma}
\newtheorem{cor}[thm]{Corollary}
\crefname{thm}{Theorem}{Theorems}
\crefname{prop}{Proposition}{Propositions}
\crefname{lem}{Lemma}{Lemmas}
\crefname{cor}{Corollary}{Corollaries}

\theoremstyle{definition}
\newtheorem{defn}[thm]{Definition}
\crefname{defn}{Definition}{Definitions}

\theoremstyle{remark}
\newtheorem{rem}[thm]{Remark}
\newtheorem{fact}[thm]{Fact}
\crefname{rem}{Remark}{Remarks}
\crefname{exa}{fact}{Fact}

\AddToHook{env/maincor/begin}{\crefalias{mainthm}{maincor}}
\AddToHook{env/prop/begin}{\crefalias{thm}{prop}}
\AddToHook{env/lem/begin}{\crefalias{thm}{lem}}
\AddToHook{env/cor/begin}{\crefalias{thm}{cor}}
\AddToHook{env/defn/begin}{\crefalias{thm}{defn}}
\AddToHook{env/rem/begin}{\crefalias{thm}{rem}}
\AddToHook{env/fact/begin}{\crefalias{thm}{fact}}

\title{Monotonicity and a Taylor approximation theorem for transseries}
\author{Vincenzo Mantova}
\date{8 May 2026}
\thanks{Supported by EPSRC grant EP/T018461/1. For the purpose of open access, the authors has applied a creative commons attribution (CC BY) licence to any author accepted manuscript version arising. No data is associated with the article.}

\begin{document}

\begin{abstract}
  We show that the composition of omega-series by surreal numbers, or more generally by elements of any confluent field of transseries, is monotonic in its second argument. In particular, omega-series and LE-series interpreted as functions have the intermediate value property. We also deduce a Taylor approximation theorem for omega-series with maximal radius of validity.
\end{abstract}

\maketitle

\addtocontents{toc}{\protect\setcounter{tocdepth}{0}}
\section{Introduction}

Transseries are a generalisation of formal power series that provide a framework for asymptotic analysis of certain classes of non-oscillating germs, one notable example being first return maps in Hilbert's sixteenth problem. There are a few different flavours in the literature, such as LE-series and grid-based series (an incomplete list is~\cite{DMM1997,Hoe1997,Kuh2000,Schmeling2001}), but for the most generality we will work with \emph{omega-series} from~\cite{BM2019}, which are a little easier to define, and contain isomorphic copies of most other flavours.

Call \textbf{field of transseries} $\Tb$ a subfield of some (restricted) Hahn field
\[ \Tb = \Rb\fspc{\Mf}_{\On} = \left\{ \sum_{i < \alpha} r_i \mf_i : \alpha \text{ ordinal}, r_i \in \Rb^{\neq 0}, \mf_i \in \Mf \ \text{with}\ i < j \to \mf_i > \mf_j \right\}, \]
where the \textbf{monomials} $\Mf$ form an ordered abelian group equipped with an isomorphism $\log : (\Mf,\cdot) \simeq \Jb_{\Mf} \coloneqq (\Rb\fspc{\Mf^{>1}}_\On,+)$ satisfying ${\log(\mf)}^n < \mf$ for every $n \in \Nb$ (we call its inverse $\exp$ and abbreviate $e^\gamma = \exp(\gamma)$). We warn the reader that such $\Mf$ must be a proper class; to avoid this issue, one can require, for instance, that the above ordinals $\alpha$ are countable. Plenty such fields exist, one notable example being Conway's field of surreal numbers $\No$.

Now suppose that $\Mf$ contains a \textbf{log-atomic} element $T \in \Mf^{>1}$, namely such that $\log^{\circ n}(T) \in \Mf$ for all $n \in \Nb$, where $\log^{\circ n}$ is the $n$-fold composition of $\log$. Then the field $\OM$ of \textbf{omega-series} in $T$ is the smallest subclass of $\Tb$ that contains $\Rb$, $T$ and is closed under taking $\exp$, $\log$, and sums like $\sum_{i < \alpha} r_i \mf_i$. $\OM$ is automatically of the form $\Rb\fspc{\Of}_\On$ for some subgroup $\Of \leq \Mf$. The choice of $T$ and of ambient $\Tb$ is irrelevant: all fields of omega-series are isomorphic to each other as fields of transseries.

By construction, $\OM$ contains, for instance, any formal Laurent series in $T^{-1}$, such as $T + 1 + T^{-1} + T^{-2} + \ldots$, but also series containing exponential terms, such as the asymptotic expansion of the $\Gamma$ function at $+\infty$:
\[ \sqrt{2\pi} e^{T\log(T) - T - \frac{1}{2}\log(T)}\left(1 + \frac{1}{12T} + \frac{1}{288T^2} + \cdots \right). \]
By inspecting the constructions in the literature, one can easily verify that $\OM$ also contains (unique) copies of grid-based transseries (\cite{Hoe1997,Hoe2006}), of $\LE$-series (\cite{DMM1997,DMM2001}), and even $\mathrm{EL}$-series (\cite{Kuh2000}). One cannot forget mentioning that $\LE$-series and grid-based transseries have been subject of intense model theory investigations and, as differential ordered valued fields, they are model complete and have the same theory as any maximal Hardy field~\cite{ADH2017,ADH2024}.

Note that $\log$ can be extended naturally to a \emph{global} function on the positive transseries: given $f \in \Tb^{>0}$, it can we written uniquely as $f = r\mf(1 + \varepsilon)$, where $r\mf$ is the \textbf{leading term} of $f$, with $r \in \Rb^{>0}$ and $\mf \in \Mf$, and $|\varepsilon| < \Rb^{>0}$, and we can define
\[ \log(f) = \log(r\mf(1 + \varepsilon)) = \log(\mf) + \log(r) + \sum_{n=1}^\infty {(-1)}^{n+1}\frac{\varepsilon^n}{n} \]
where $\log(r)$ is standard real logarithm, and the infinite sum is always well defined, although we gloss over the details of what that means. Its inverse $\exp$ has a similar global extension given by a Taylor series. Since $\OM$ is closed under infinite sums, it is also closed under global $\exp$ and $\log$.

Just like traditional power series can be differentiated and composed, the same is true for omega-series. There is a unique derivation $f \mapsto f'$ on $\OM$ (\cite{Schmeling2001,BM2018}) satisfying $T' = 1$, $(e^\gamma)' = e^\gamma \gamma'$ for any $\gamma \in \Jb$, and which is \textbf{strongly $\Rb$-linear}, namely
\[ {\left(\sum_{i < \alpha} r_i \mf_i\right)}' = \sum_{i < \alpha} r_i \mf_i'. \]

Likewise, given $g \in \OM$ with $g > \Rb$, there is a unique right composition map $f \mapsto f \circ g$ satisfying $T \circ g = g$, $e^\gamma \circ g = e^{\gamma \circ g}$ for $\gamma \in \Jb$, and which is strongly $\Rb$-linear
\[ {\left(\sum_{i < \alpha} r_i \mf_i\right)} \circ g = \sum_{i < \alpha} r_i (\mf_i \circ g). \]
More generally, let $\Ub = \Rb\fspc{\Nf}_{\On}$ be a \textbf{confluent} field of transseries, namely such that for every $x \in \Ub^{>\Rb}$ there is $n \in \Nb$ such that the leading term of $\log^{\circ n}(x)$ is log-atomic (this holds in $\OM$, as by construction, $\log^{\circ n}(x)$ has leading monomial of the form $\log^{\circ k}(T)$ for $n$ sufficiently large). Note that $\Ub$ may be a new field, different from $\Tb$, although it will also contain copies of $\OM$ obtained by replacing $T$ with log-atomic monomials of $\Ub$. Then for every $x \in \Ub^{>\Rb}$ there is a unique map $f \mapsto f \circ x$ satisfying the above properties with $x$ in place of $g$ (while the uniqueness is obvious, the existence is technically challenging, see~\cite{Schmeling2001,BM2019}; while~\cite{BM2019} is stated for $\No$, and seemingly using the assumption `T4', the proof therein only uses confluence). Surreal numbers are an example of confluent field (\cite{BM2018}).

Finally, we remark that composition and derivation are compatible with global $\exp$, namely $\exp(f)' = \exp(f)f'$ and $\exp(f) \circ g = \exp(f \circ g)$ for any $f \in \OM$, and also with each other, in the sense of the \textbf{chain rule}: for any $f, g \in \OM$ with $g > \Rb$ we have $(f \circ g)' = (f' \circ g)g'$. Moreover, composition is compatible `with itself' in the sense that it is associative: $(f \circ g) \circ x = f \circ (g \circ x)$ for any $f \in \OM$, $g \in \OM^{>\Rb}$, $x \in \Ub^{>\Rb}$.

It is immediate from the definition that \emph{right} composition by some $x \in \Ub^{>\Rb}$ must be an ordered field embedding: by construction, the map $f \mapsto f \circ x$ preserves (infinite) sums, scalar multiplication by $\Rb$, $\exp$, $\log$, and it quickly follows that it preserves multiplication and ordering.

However, the properties of \emph{left} composition by some $f \in \OM$, namely of the map $x \mapsto f \circ x$, have not been explored as much. Since transseries have been introduced to provide asymptotic expansions of real functions, we expect these maps to share at least the most basic properties. For instance, we expect the map to be strictly increasing exactly when $f' > 0$. To the best of my knowledge, this only appears in a rarely cited preprint of Edgar~\cite[Prop.\ 4.9]{Edg2009}, for $\LE$-series only, under the heading `Simpler proof needed', and with an `overly-involved proof'.

Addressing Edgar's need, this note gives a short, self-contained proof of monotonicity for general omega-series. We also spell out a small set of inductive assumptions, so that the method can be reapplied more easily to fields larger than $\OM$, in particular with an eye to the hyperseries of~\cite{BHK2021}, which generalise transseries by adding transexponential functions. From now on, let $\Ub$ be a fixed confluent field of transseries, for instance $\Ub = \No$, or more generally a field of transseries equipped with a left composition by omega-series that preserves infinite sums, scalar multiplication by $\Rb$, $\exp$, and $\log$.

\begin{mainthm}[Monotonicity]\label{monotonicity}
  For all $f \in \OM$, the function $x \mapsto f \circ x$ for $x \in \Ub^{>\Rb}$ is strictly increasing if $f' > 0$, strictly decreasing if $f' < 0$, and constant if $f' = 0$.
\end{mainthm}

Since $\OM$ is generated from $T$, a naive proof by induction would be as follows: the conclusion is true for the base case $f = T$, $f = \log(T)$, \ldots; if $f = \sum_{i < \alpha} r_i e^{\gamma_i}$, we are free to assume the conclusion for each map $x \mapsto \gamma_i \circ x$, and we would like to deduce that $x \mapsto f \circ x$ is also monotonic. This works very well for $f \in \Jb_\Of^{>0}$, and we do so in \Cref{monotonicity-J}, but the inductive assumption falls short for the general case. We also need the following very weak form of mean value property: given $\gamma \in \Jb_\Of^{>0}$ and $x, y \in \Ub$ with $x > y > \Rb$, we have
\[ \frac{\gamma \circ x - \gamma \circ y}{x - y} \preceq \gamma' \circ y \quad \text{at least when} \quad \gamma \circ x - \gamma \circ y \preceq 1 \text{ or } (\gamma' \circ y)(x - y) \prec 1, \]
where $a \preceq b$ means $|a| \leq n|b|$ for some $n \in \Nb$ and $a \prec b$ means $n|a| < |b|$ for all $n \in \Nb$. The condition is almost trivial to verify in the base case of $T$, $\log(T)$, $\ldots$, as the associated functions $x \mapsto \log^{\circ n}(x)$ are all strictly increasing and \emph{concave}. The rest of the induction is then relatively straightforward (see \Cref{weak-mvp}), and it leads to a fairly short proof of \Cref{monotonicity}.

Since omega-series admit compositional inverses by~\cite{Bag2025}, namely for every $f \in \OM^{>\Rb}$ there is a (unique) $f^{\operatorname{inv}} \in \OM^{>\Rb}$ such that $f \circ f^{\operatorname{inv}} = f^{\operatorname{inv}} \circ f = T$, we can immediately deduce the following.

\begin{maincor}\label{ivp}
  For all $f \in \OM$, the function $x \mapsto f \circ x$ for $x \in \Ub^{>\Rb}$ has the intermediate value property: if $f \circ x \leq w \leq f \circ y$, there is $z \in \Ub$ between $x$ and $y$ such that $f \circ z = w$.
\end{maincor}

$\LE$-series are closed under composition and admit compositional inverses within $\LE$ (\cite{DMM2001}), thus the intermediate value property applies even for $x \mapsto f \circ x$ seen as a function from $\LE^{>\Rb}$ to $\LE$.

Given \Cref{monotonicity}, it also becomes relatively easy to prove the following form of Taylor's theorem. We write $a \asymp b$ to mean $a \preceq b$ and $b \preceq a$, and $O(b)$ to represent an unspecified element of the set $\{a : a \preceq b\}$.

\begin{mainthm}[Taylor approximation for $\OM$]\label{taylor}
  Let $f \in \OM$ be non-zero with $f \not\asymp T^k$ for all $k \in \Nb$, and let $x,\delta \in \Ub$ with $x > \Rb$, $x + \delta > \Rb$, $x + \delta \asymp x$, and $\frac{f' \circ x}{f \circ x}\delta \preceq 1$. Then for all $n \geq 0$,
  \[ f \circ (x + \delta) = \sum_{i = 0}^{n - 1} \frac{f^{(i)} \circ x}{i!}\delta^i + O\left((f^{(n)} \circ x)\delta^n\right). \]
\end{mainthm}

The restriction $f \not\asymp T^k$ is a mere technicality: since the derivatives of $T^k$ vanish after $k + 1$ steps, one must look at $f^{(k+1)}$ to determine the radius of validity of the approximation, which results in a slightly more complicated but functionally equivalent statement (see \Cref{taylor-general}).

Other versions of the Taylor theorem exist in the literature, but they focus on the \emph{equality} $f \circ (x + \delta) = \sum_{i = 0}^\infty \frac{f^{(i)} \circ x}{i!}\delta^i$, which in general holds for a much smaller class of $\delta$'s (see for example~\cite[\S 6]{DMM2001},~\cite[Prop. 7.13]{BM2019}). What is notable here is that the Taylor approximation is valid for the essentially largest possible meaningful class of $\delta$'s. Indeed, when $\delta$ is small as required by \Cref{taylor}, the sequence $(f^{(n)} \circ x)\delta^n$ is \emph{weakly decreasing} for $\prec$, and even \emph{strictly} decreasing if we require $\delta \prec x$ and $(f^\dagger \circ x)\delta \prec 1$, whereas it is \emph{strictly increasing} if $\delta \succ x$ or $(f^\dagger \circ x)\delta \succ x$, making the error terms larger rather than smaller (see \Cref{radius-of-convergence}). The conclusion may hold or fail in the remaining case $x + \delta \prec x$ depending on $f$ (\Cref{validity-at-boundary}). In separate work with Bagayoko~\cite{BM2025}, we tackle the equality $f \circ (x + \delta) = \sum_{i = 0}^\infty \frac{f^{(i)} \circ x}{i!}\delta^i$ for omega-series, and show it holds under $\delta \prec x$ and the strict inequality $(\mf^\dagger \circ x)\delta \prec 1$ with respect to every monomial $\mf$ appearing in $f$, rather than just $f$, expanding and generalising the numerous existing variants of the same result.

Several questions remain open. Notably, while we have seen that omega-series and $\LE$-series, when interpreted as functions, have the intermediate value property, we are not claiming that they have the \emph{mean value property}: given $z$ such that $f' \circ z = \frac{f \circ x - f \circ y}{x - y}$, \Cref{monotonicity} does not seem to imply directly that $z$ is between $x$ and $y$. This is related to whether $f'' \geq 0$ implies that the function $x \mapsto f \circ x$ is \emph{convex}, namely that $f' \circ y \leq \frac{f \circ x - f \circ y}{x - y}$ for $x > y$, and it may require a new argument.

Going beyond omega-series, it would be interesting to know if composition of elements in \emph{hyperserial fields} by \emph{hyperseries} (\cite{BHK2021}) is also monotonic as in \Cref{monotonicity}, and deduce the corresponding generalisation of \Cref{taylor}.

In this context, the first question is whether the weak mean value property used in the proof holds for the basic hyperlogarithmic monomials `$\ell_\alpha$', which represent, in a suitable sense, the `$\alpha$-th iterates of $\log$'. While the condition is satisfied locally, by the local Taylor expansion of the transfinite iterates of $\log$, it does not seem obvious at a first glance that it has to hold globally. For the usual $\log$, there is no issue because of the functional equation $\log(xy) = \log(x) + \log(y)$, but from $\ell_\omega$ onwards functions are less well behaved. In principle, this may require additional restrictions on the growth properties in the definition of hyperserial fields, and one should check if they are already satisfied in the surreal model of~\cite{BH2023}.

A similar, and potentially equivalent consideration emerges from \Cref{monotonicity-J} (monotonicity for $\gamma \in \Jb^{>0}$): its proof relies on the fact that for $x > y > \Rb$ and $z > w > \Rb$, if $y > w$ and $x - y > z - w$, then $\frac{\exp(x)}{\exp(y)} > \frac{\exp(z)}{\exp(w)} > 1$. Once again, this property seems easy to verify locally for transfinite iterations of $\exp$, but it is less clear whether it can be extended globally, due to the lack of a functional equation like $\exp(x + y) = \exp(x)\exp(y)$.

\subsection*{Acknowledgements} I thank Dino Peran, Jean-Philippe Rolin, and Tamara Servi for inviting me to work on normal forms. \Cref{taylor} was formulated only after it became clear it was the right statement to handle normal forms of omega-series. I also thank Vincent Bagayoko for the numerous helpful discussions on transseries and hyperseries, and for this note in particular, for helping me clarify that the bounds in \Cref{taylor} are sharp (see \Cref{iter-der-dagger}). I also thank an anonymous referee for the careful reading, which led to the removal of redundant assumptions in \Cref{monotonicity} and \Cref{ivp}. Part of this research was carried out while the author was at the IHP trimester `Model theory, combinatorics and valued fields' in 2018 with support from the \foreignlanguage{french}{Fondation Sciences Mathématiques de Paris-FSMP}. The author was also supported by the Engineering and Physical Sciences Research Council grant EP/T018461/1.

\tableofcontents

\addtocontents{toc}{\protect\setcounter{tocdepth}{1}}
\section{Some preliminaries}

Below are terminology, notations, and facts that will be used throughout.

As anticipated in the introduction, we shall use the following notation for the \textbf{dominance relation}: for $a,b$ in any ordered ring, we write
\begin{itemize}
  \item $a \preceq b$ if $|a| \leq n|b|$ for some $n \in \Nb$ (a total partial order);
  \item $a \asymp b$ if $a \preceq b$ and $b \preceq a$ (an equivalence relation);
  \item $a \prec b$ if $a \preceq b$ and $b \not\preceq a$; equivalently, $n|a| < |b|$ for all $n \in \Nb$ (a strict partial order);
  \item $a \sim b$ if $a - b \prec a$ (an equivalence relation on the non-zero elements);
  \item $O(a)$ represents the convex class $\{b : b \preceq a\}$, and similarly $o(a)$ represents $\{b : b \prec a\}$; both shall be used as in the big $O$ notation.
\end{itemize}

\begin{rem}
  Right composition by a fixed $x \in \Ub^{>\Rb}$ yields an ordered exponential field embedding. In particular, we also have for instance that $f \prec g$ holds if and only if $f \circ x \prec g \circ x$, and as a special case, $f \prec T$ if and only if $f \circ x \prec x$. The same holds for all of the above relations, since they are solely defined on the basis of the underlying ordered field structure. This will be used liberally in the proofs.
\end{rem}

\begin{fact}
  As an ordered differential field, $\OM$ is an $H$-field, namely $f > \Rb$ implies $f' > 0$, otherwise $f = r + \varepsilon$ where $r' = 0$ (in fact, $r \in \Rb$) and $|\varepsilon|$ is smaller than all the constants (that is, $\varepsilon \prec 1$). This has numerous consequences, but the reader will only need to know that:
  \begin{itemize}
    \item if $1 \not\asymp f$, then $f \succeq g$ if and only if $f' \succeq g'$, and if $f \succ g$ if and only if $f' \succ g'$;
    \item if $1 \not\asymp f$, then $f \asymp g$ if and only if $f' \asymp g'$;
    \item if $1 \not\asymp f$, then $f \sim g$ if and only if $f' \sim g'$;
    \item if $f \preceq 1$, then $f' \prec 1$;
    \item if $f \prec 1$, $g \neq 0$, $g \not\asymp 1$, then $f' \prec \frac{g'}{g}$.
  \end{itemize}
\end{fact}

Since $\OM$ is generated by $T$, most arguments use induction on how elements are constructed starting from $T$. We formalise this with the following rank. Since there is no risk of ambiguity, we shall abbreviate $\Jb_\Of = \Rb\fspc{\Of^{>1}}_\On$ with just $\Jb$.

\begin{defn}
  For any $f = \sum_{i < \alpha} r_i e^{\gamma_i} \in \OM$, where each $r_i$ is a non-zero real number and $\gamma_i \in \Jb$, we define the \textbf{exponential rank} $\ER(f)$ of $f$ to be the ordinal:
  \begin{itemize}
    \item $0$ if $f$ is a monomial of the form $\log^{\circ n}(T)$ for some $n \in \Nb$, or if $f = 0$;
    \item $\sup\{\ER(\gamma_i) + 1 : i < \alpha \}$ otherwise.
  \end{itemize}
\end{defn}
This is clearly well defined (see~\cite{BM2019} for more details).

\begin{rem}
  It is immediate from the definition that for $f, g \in \OM$ we have $\ER(f + g) \leq \max\{\ER(f),\ER(g)\}$, unless $\ER(f) = \ER(g) = 0$, in which case $\ER(f + g) \leq 1$. Similarly, $\ER(-f) \leq \ER(f)$ unless $\ER(f) = 0$, in which case $\ER(-f) = 1$. In particular $\ER(f - g) \leq \max\{\ER(f),\ER(g)\}$ unless $\ER(f) = \ER(g) = 0$.
\end{rem}

We do not define Hahn fields here, but we refer the reader to any of the cited sources about transseries for details about the definition of sum, product, and order on them. We just remind the reader that the set of monomials appearing in a series $f = \sum_{i < \alpha} r_i \mathfrak{m}_i$, meaning $\{\mathfrak{m}_i : i < \alpha\}$, is called \textbf{support} of $f$ (note that by how we defined fields of transseries, the support of a single series is always a \emph{set} even if the monomials range in a proper class). When $f \neq 0$, we call the maximum $\mathfrak{m}_0$ of the support the \textbf{leading monomial} of $f$, and we call $r_0\mathfrak{m}_0$ the \textbf{leading term} of $f$. Note that by construction, $f \sim r_0\mathfrak{m}_0$.

For clarity, we also remark again that $\exp$ and $\log$ have the following Taylor expansions for any $\varepsilon \prec 1$ in $\Ub$:
\begin{align*}
  \log(1 + \varepsilon) &= \sum_{n=1}^\infty {(-1)}^{n+1} \frac{\varepsilon^n}{n}, \\
  \exp(\varepsilon) &= \sum_{n=1}^\infty \frac{\varepsilon^n}{n!}.
\end{align*}
The infinite sum on the right is not a limit with respect to the topology induced by the order, but an algebraic operation in which each power $\varepsilon^n$ is expanded into a series, and then all series are summed term by term. For the details of how this is done, and why it is well defined, we defer again to the bibliography. Recall that $\OM$ is closed under infinite sums, and so in particular under the above ones.

\section{Monotonicity for purely infinite series}

We first prove monotonicity in the easier case of purely infinite series. This argument is a straightforward naive induction.

\begin{prop}\label{monotonicity-J}
  For all $\gamma \in \Jb^{>0}$, the map $x \mapsto \gamma \circ x$ is strictly increasing.
\end{prop}
\begin{proof}
  Fix some  $x, y \in \Ub$ with $x > y > \Rb$. We prove by induction on $\ER(\gamma)$ that $\gamma \circ x > \gamma \circ y$. Since by construction $0 < \log(x) - \log(y) < x - y$, by iterating $\log$, we find that for all $k > n > 0$
  \[ 0 < \log^{\circ k}(x) - \log^{\circ k}(y) < \log^{\circ n}(x) - \log^{\circ n}(y). \]
  Therefore, the statement is true for $\ER(\gamma) = 0$, and even for $\gamma = \delta - \eta$ when $\ER(\delta) = \ER(\eta) = 0$ and $\delta > \eta$.

  Now let $\gamma \in \Jb^{>0}$ have rank $\ER(\gamma) > 0$. We assume by induction that the conclusion holds for any $\delta$ such that $\ER(\delta) < \ER(\gamma)$. Let $re^{\delta}$ be the leading term of $\gamma$. Recall that $\ER(\delta) < \ER(\gamma)$.

  We claim that $\gamma \circ x - \gamma \circ y \sim re^{\delta \circ x} - re^{\delta \circ y}$. Let $e^{\eta}$ be any monomial in the support of $\gamma$ distinct from $e^{\delta}$ (if there is no such monomial, then the conclusion is obvious as $\gamma = re^\delta$). Again $\ER(\eta) < \ER(\gamma)$, and moreover $\delta > \eta > 0$, since $\gamma \in \Jb$. We have one of $\ER(\delta) = \ER(\eta) = 0$ or $\ER(\delta - \eta) \leq \max\{\ER(\delta),\ER(\eta)\} < \ER(\gamma)$, hence by inductive hypothesis
  \[ (\delta-\eta) \circ x > (\delta-\eta) \circ y, \quad \text{that is} \quad \delta \circ x - \delta \circ y > \eta \circ x - \eta \circ y. \]
  Moreover, $\eta \circ x - \eta \circ y > 0$, again by inductive hypothesis. Therefore,
  \[ e^{\delta \circ x - \delta \circ y} - 1 > e^{\eta \circ x - \eta \circ y} - 1 > 0, \quad \text{hence} \quad e^{\delta \circ x - \delta \circ y} - 1 \succeq e^{\eta \circ x - \eta \circ y} - 1, \]
  since $\exp$ is an increasing function and $e^a > 1$ for $a > 0$.

  Since $(\delta - \eta) \circ y > \Rb$, we have $e^{(\delta - \eta) \circ y} > \Rb$, or in other words, $e^{\delta \circ y} \succ e^{\eta \circ y}$. Combining the inequalities together,
  \[ e^{\delta \circ x} - e^{\delta \circ y} = e^{\delta \circ y}(e^{\delta \circ x - \delta \circ y} - 1) \succ e^{\eta \circ y}(e^{\eta \circ x - \eta \circ y} - 1) = e^{\eta \circ x} - e^{\eta \circ y}. \]
  It follows at once that $\gamma \circ x - \gamma \circ y \sim re^{\delta \circ x} - re^{\delta \circ y}$.

  Granted the claim, by inductive hypothesis, $\delta \circ x > \delta \circ y$. Recall moreover that $r > 0$, since $\gamma > 0$. It follows that $re^{\delta \circ x} - re^{\delta \circ y} > 0$, hence $\gamma \circ x - \gamma \circ y > 0$, as desired.
\end{proof}

In the course of the proof, we have also proved the following.
\begin{prop}\label{leading-gap-J}
  For all $\delta, \eta \in \Jb^{>0}$ with $\delta > \eta$ and all $x, y \in \Ub$ with $x > y > \Rb$, we have $e^{\delta \circ x} - e^{\delta \circ y} \succ e^{\eta \circ x} - e^{\eta \circ y}$.
\end{prop}

\section{Proof of monotonicity}

\begin{lem}\label{weak-mvp}
  For all $\gamma \in \Jb$ and all $x, y \in \Ub$ with $x > y > \Rb$, we have:
  \begin{enumerate}
  \item\label{weak-mvp-small} if $\gamma \circ x - \gamma \circ y \preceq 1$, then $\frac{\gamma \circ x - \gamma \circ y}{x - y} \preceq \gamma' \circ y$,
  \item\label{weak-mvp-large} if $\gamma \circ x - \gamma \circ y \succeq 1$, then $(\gamma' \circ y)(x - y) \succeq 1$.
  \end{enumerate}
\end{lem}
\begin{proof}
  Fix $x$, $y$ as in the assumptions. Let $\gamma \in \Jb$. We shall prove the conclusion by induction on $\ER(\gamma)$.

  When $\gamma$ is $\log^{\circ n}(T)$ for some $n$, both conclusions hold because $\log^{\circ n}$ is concave: $0 < \frac{\log^{\circ n}(a) - \log^{\circ n}(b)}{a - b} \leq (\log^{\circ n})'(b)$ for any $b$ in its domain and $a > b$, and so also for $x > y > \Rb$. The conclusion is trivial for $\gamma = 0$.

  Now assume $\ER(\gamma) > 0$ and let $re^\delta$ be the leading term of $\gamma$. Recall that $\delta \in \Jb^{>0}$ and that $\ER(\delta) < \ER(\gamma)$. By \Cref{leading-gap-J}, we have
  \[ \gamma \circ x - \gamma \circ y \sim r(e^{\delta \circ x} - e^{\delta \circ y}) = re^{\delta \circ y}(e^{\delta \circ x - \delta \circ y} - 1), \]
  which we rewrite as
  \[ \frac{\gamma \circ x - \gamma \circ y}{x - y} \sim re^{\delta \circ y}\frac{e^{\delta \circ x - \delta \circ y} - 1}{x - y}. \]
  Recall moreover that since $\gamma \sim re^\delta$ and $\gamma \not\asymp 1$, we have $\gamma' \sim re^\delta \delta'$, so in particular $\gamma' \circ y \sim re^{\delta \circ y} (\delta' \circ y)$.

  We distinguish two cases. If $\delta \circ x - \delta \circ y \prec 1$, then we can use the Taylor expansion of $\exp$ to check that
  \[ \frac{e^{\delta \circ x - \delta \circ y} - 1}{x - y} \sim \frac{\delta \circ x - \delta \circ y}{x - y} \preceq \delta' \circ y, \]
  where the second inequality follows from the inductive hypothesis~\eqref{weak-mvp-small}. Therefore,
  \[ \frac{\gamma \circ y - \gamma \circ y}{x - y} \sim re^{\delta \circ y}\frac{e^{\delta \circ x - \delta \circ y} - 1}{x - y} \preceq re^{\delta \circ y}(\delta' \circ y) \sim \gamma' \circ y. \]
  Both conclusions follow trivially.

  Now suppose that $\delta \circ x - \delta \circ y \succeq 1$. Note in particular that $\gamma \circ x - \gamma \circ y \succ 1$, because $e^{\delta \circ y} \succ 1$ and $e^{\delta \circ x - \delta \circ y} - 1 \succeq 1$, so we only need to prove conclusion~\eqref{weak-mvp-large}. By inductive hypothesis~\eqref{weak-mvp-large}, $(\delta' \circ y)(x - y) \succeq 1$, which with $e^{\delta \circ y} \succ 1$ yields~\eqref{weak-mvp-large}:
  \[ (\gamma' \circ y)(x - y) \sim re^{\delta \circ y}(\delta' \circ y)(x - y) \succ 1. \qedhere \]
\end{proof}

\begin{cor}\label{small-gaps}
  For all $\gamma \in \Jb^{<0}$ and all $x, y \in \Ub$ with $x > y > \Rb$, we have
  \[ \frac{e^{\gamma \circ x} - e^{\gamma \circ y}}{x - y} \preceq e^{\gamma \circ y}(\gamma' \circ y) = (e^\gamma)' \circ y. \]
\end{cor}
\begin{proof}
  Write
  \[ e^{\gamma \circ x} - e^{\gamma \circ y} = e^{\gamma \circ y}(e^{\gamma \circ x - \gamma \circ y} - 1). \]
  If $\gamma \circ x - \gamma \circ y \prec 1$, then using \Cref*{weak-mvp}\eqref{weak-mvp-small}
  \[ e^{\gamma \circ y}(e^{\gamma \circ x - \gamma \circ y} - 1) \sim e^{\gamma \circ y}(\gamma \circ x - \gamma \circ y) \preceq e^{\gamma \circ y}(\gamma' \circ y)(x - y). \]
  Otherwise, $e^{\gamma \circ x - \gamma \circ y} \not\sim 1$, and since $\gamma \circ x < \gamma \circ y$ by \Cref{monotonicity-J} applied to $-\gamma$, we have $e^{\gamma \circ x - \gamma \circ y} \preceq 1$, thus
  \[ e^{\gamma \circ y}(e^{\gamma \circ x - \gamma \circ y} - 1) \asymp e^{\gamma \circ y} \preceq e^{\gamma \circ y}(\gamma' \circ y)(x - y), \]
  where the last inequality follows from \Cref*{weak-mvp}\eqref{weak-mvp-large}.
\end{proof}

\begin{proof}[Proof of \Cref{monotonicity}]
  Let $x, y \in \Ub$ with $x > y > \Rb$ and $f \in \OM$. Trivially, when $f' = 0$, we have $f = r \in \Rb$, so $r \circ x = r \circ y = r$, as desired. We now assume that $f' > 0$, and we want to prove $f \circ x > f \circ y$. The case $f' < 0$ will follow trivially by replacing $f$ with $-f$.

  \textbf{Case $f = T + \varepsilon$ with $\varepsilon \prec T$.} Note that in this case $f' \sim 1$, so $f' > 0$. We claim that $\varepsilon \circ x - \varepsilon \circ y \prec x - y$, and so $f \circ x - f \circ y \sim x - y$, which clearly implies $f \circ x > f \circ y$.

  Suppose that $e^\gamma$ is a monomial in the support of $\varepsilon$. We claim that $e^{\gamma \circ x} - e^{\gamma \circ y} \prec x - y$. This is trivial for $\gamma = 0$, and an immediate consequence of \Cref{leading-gap-J} for $\gamma > 0$. For $\gamma < 0$, we have $1 \succ (e^{\gamma})'$, so $1 \succ (e^{\gamma})' \circ y$, hence by \Cref{small-gaps}
  \[ e^{\gamma \circ x} - e^{\gamma \circ y} \preceq e^{\gamma \circ y}(\gamma' \circ y)(x - y) = ((e^\gamma)' \circ y)(x - y) \prec x - y. \]
  Taking the sum over all terms in $\varepsilon$, we find that $\varepsilon \circ x - \varepsilon \circ y \prec x - y$.

  \textbf{Case $f > \Rb$.} Recall that $\OM$ is confluent, because for any $g \in \OM^{>\Rb}$, the leading monomials along the sequence $(\log^{\circ n}(g) : n \in \Nb)$ must have decreasing exponential rank, until one reaches $\log^{\circ k}(T)$ for some $k$. Pick $n \in \Nb$ such that $\log^{\circ n}(T) \circ f = \log^{\circ n}(f) = \log^{\circ k}(T) + \varepsilon$ for some $k \in \Nb$ and some $\varepsilon \prec \log^{\circ k}(T)$. Note that
  \[ \log^{\circ n}(T) \circ f \circ \exp^{\circ k}(T) = T + \overline{\varepsilon} \]
  where $\overline{\varepsilon} = \varepsilon \circ \exp^{\circ k}(T) \prec \log^{\circ k}(T) \circ \exp^{\circ k}(T) = T$. By the previous case, the function $x \mapsto x + (\overline{\varepsilon} \circ x)$ is strictly increasing. Since the functions $\log$ and $\exp$ are also strictly increasing on $\Ub^{>\Rb}$, then so is the function $x \mapsto f \circ x$.

  \textbf{Case $f \not > \Rb$}. Since by assumption $f' > 0$, we cannot have $f < \Rb$, so there must be some $r \in \Rb$ such that $f - r \prec 1$, and moreover $f - r < 0$. It follows that $g = -\frac{1}{f - r} > \Rb$. By the previous case, the function $x \mapsto g \circ x$ is strictly increasing, thus so is the function
  \[ x \mapsto -\frac{1}{g \circ x} + r = f \circ x. \qedhere \]
\end{proof}

\begin{cor}\label{shrinking-gap}
  For all $f, g \in \OM$ with $1 \not\asymp f \succ g$ and all $x,y \in \Ub$ with $x > y > \Rb$, we have
  \[ f \circ x - f \circ y \succ g \circ x - g \circ y. \]
\end{cor}
\begin{proof}
  Up to replacing $f$ with $-f$, we may assume that $f' > 0$. Since $f \not\asymp 1$, we have $f' \succ g'$, thus $f' - rg' = (f - rg)' > 0$ for every $r \in \Rb$. By \Cref{monotonicity},
  \[ (f - rg) \circ x > (f - rg) \circ y, \quad \text{hence} \quad f \circ x - f \circ y > r(g \circ x - g \circ y). \]
  Moreover, $f \circ x - f \circ y > 0$, and the conclusion follows.
\end{proof}

\begin{cor}\label{leading-gap}
  For all $f, g \in \OM$ with $1 \not\asymp f \sim g$ and all $x,y \in \Ub$ with $x > y > \Rb$, we have
  \[ f \circ x - f \circ y \sim g \circ x - g \circ y. \]
\end{cor}
\begin{proof}
  Apply \Cref{shrinking-gap} to $f$ and $f - g \prec f$.
\end{proof}

\begin{proof}[Proof of \Cref{ivp}]
  Fix $f \in \OM$, $x, y \in \Ub^{>\Rb}$, $w \in \Ub$ such that $f \circ x \leq w \leq f \circ y$. Note that the case $f \in \Rb$ is trivial, so assume $f \notin \Rb$. There is $z \in \Ub^{>\Rb}$ such that $f \circ z = w$: when $f > \Rb$, just take $z = f^{\mathrm{inv}} \circ w$, where $f^{\mathrm{inv}} \in \OM^{>\Rb}$ is the compositional inverse of $f$ (\cite{Bag2025}); otherwise, replace $f$ with $\pm\frac{1}{f - r} > \Rb$ just as in the proof of \Cref{monotonicity}, and reduce to $f > \Rb$. By \Cref{monotonicity}, $z$ has to be between $x$ and $y$.
\end{proof}

\section{Taylor's approximation}

In what follows, fix some $x, \delta \in \Ub$ with $x > \Rb$ and $x + \delta > \Rb$. We shall abbreviate $\frac{f'}{f}$ with $f^\dagger$, provided $f \neq 0$. Note that $f^\dagger = (\log|f|)'$.

\begin{prop}\label{taylor-first-term}
  For all non-zero $f \in \OM$, if $f \not\asymp 1$, $x + \delta \asymp x$, and $(f^\dagger \circ x)\delta \preceq 1$, then
  \[ f \circ (x + \delta) \asymp f \circ x \quad \text{and} \quad f \circ (x + \delta) - f \circ x \asymp (f' \circ x)\delta. \]
  In particular, if $(f^\dagger \circ x)\delta \prec 1$, then $f \circ (x + \delta) \sim f \circ x$.
\end{prop}
\begin{proof}
  Note that $x + \delta \asymp x$ implies $\delta \preceq x$. We work by induction on $\ER(f)$. For the base case, suppose that $f = \log^{\circ n}(T)$ for some $n \in \Nb$. Both conclusions are trivial for $n = 0$. For $n = 1$, we know that
  \[ \log(x + \delta) - \log(x) = \log\left(1 + \frac{\delta}{x}\right) \asymp \frac{\delta}{x} = (f' \circ x)\delta. \]
  Note that $\frac{\delta}{x} \preceq 1 \prec \log(x)$, so in particular $\log(x + \delta) \sim \log(x)$. For $n > 1$, recall that the Taylor series of $\log$ implies that $\log(y + \varepsilon) - \log(y) \sim \frac{\varepsilon}{y}$ when $\varepsilon \prec y$. Using $y = \log^{\circ n-1}(x)$, $\varepsilon = \log^{\circ n-1}(x + \delta) - y$, we can verify inductively that $\varepsilon \prec y$, or in other words $\log^{\circ n-1}(x + \delta) \sim \log^{\circ n-1}(x)$, and
  \begin{multline*}
    \log(\log^{\circ n-1}(x + \delta)) - \log(\log^{\circ n-1}(x)) \sim \frac{\log^{\circ n-1}(x + \delta) - \log^{\circ n-1}(x)}{\log^{\circ n-1}(x)} \\
    \sim \frac{\log(x + \delta) - \log(x)}{\log(x) \cdots \log^{\circ n-1}(x)} \asymp \frac{\delta}{x\log(x) \cdots \log^{\circ n-1}(x)} = (f' \circ x)\delta.
  \end{multline*}

  For $\ER(f) > 0$, we first prove the conclusions for $f = e^\gamma$ with $\gamma \in \Jb^{\neq 0}$. Suppose that $(f^\dagger \circ x)\delta \preceq 1$; since $f^\dagger = \frac{e^\gamma \gamma'}{e^\gamma} = \gamma'$, we have $(\gamma' \circ x)\delta \preceq 1$, and in particular $(\gamma^\dagger \circ x)\delta \prec 1$, as $\gamma \succ 1$. By inductive assumption, $\gamma \circ (x + \delta) - \gamma \circ x \asymp (\gamma' \circ x)\delta \preceq 1$. Since $e^y - 1 \asymp y$ for $y \preceq 1$, we get
  \[ e^\gamma \circ (x + \delta) - e^\gamma \circ x = e^{\gamma \circ x}\left(e^{\gamma \circ (x + \delta) - \gamma \circ x} - 1\right) \asymp e^{\gamma \circ x}(\gamma' \circ x)\delta = ((e^\gamma)' \circ x)\delta. \]
  Since $e^y \asymp 1$ for $y \preceq 1$, we similarly find that
  \[ e^\gamma \circ (x + \delta) = e^{\gamma \circ x}e^{\gamma \circ (x + \delta) - \gamma \circ x} \asymp e^{\gamma \circ x} = e^\gamma \circ x. \]

  For general $f$, let $re^\gamma$ be the leading term of $f$. Note that by assumption $\gamma \neq 0$. Since $f \sim re^{\gamma}$, we have
  \[ f \circ (x + \delta) \sim re^\gamma \circ (x + \delta) \asymp re^\gamma \circ x \sim f \circ x. \]
  Since $f \not\asymp 1$, we also have $f' \sim (re^\gamma)'$ and $f^\dagger \sim {(re^\gamma)}^\dagger = \gamma'$, hence by \Cref{leading-gap},
  \[ f \circ (x + \delta) - f \circ x \sim re^{\gamma} \circ (x + \delta)  - re^{\gamma} \circ x \asymp r((e^\gamma)' \circ x)\delta \sim (f' \circ x)\delta. \]

  When $(f^\dagger \circ x)\delta \prec 1$, or in other words $(f' \circ x)\delta \prec f \circ x$, we find that indeed $f \circ (x + \delta) \sim f \circ x$.
\end{proof}

To prove Taylor's theorem, we now want to iterate the above approximation. To do that, we need to control the assumption $(f^\dagger \circ x)\delta \preceq 1$ when we replace $f$ with its derivatives.

\begin{lem}\label{iter-der-dagger}
  Let $f \in \OM$ be non-zero such that $f \not\asymp T^k$ for all $k \in \Nb$.
  \begin{enumerate}
    \item\label{iter-der-dagger-large} If $f^\dagger \succ \frac{1}{T}$, then ${(f^{(n)})}^\dagger \sim f^\dagger$ for all $n$; in particular, $f^{(n)} \sim {(f^\dagger)}^n f$.
    \item\label{iter-der-dagger-small} Otherwise, for some $\ell \in \Nb$, ${(f^{(n)})}^\dagger \asymp \frac{1}{T}$ for all $n \neq \ell - 1$, and if $\ell > 0$, then ${(f^{(\ell - 1)})}^\dagger \prec \frac{1}{T}$; in particular, $f^{(n)} \preceq \frac{f}{T^n}$, $f^{(\ell+n)} \asymp \frac{f^{(\ell)}}{T^n}$ for all $n$.
  \end{enumerate}
\end{lem}
\begin{proof}
  \eqref{iter-der-dagger-large} We have $\frac{1}{f^\dagger} = \frac{f}{f'} \prec T$, hence
  \[ T' = 1 \succ \left(\frac{f}{f'}\right)' = 1 - \frac{f f''}{{(f')}^2} = 1 - \frac{{(f')}^\dagger}{f^\dagger}. \]
  This says that ${(f')}^\dagger \sim f^\dagger$, and in particular also that ${(f')}^\dagger \succ \frac{1}{T}$. By induction, ${(f^{(n)})}^\dagger \sim f^\dagger$ for all $n \in \Nb$. Since $f^{(n+1)} = {(f^{(n)})}^\dagger f^{(n)} \sim f^\dagger f^{(n)}$, we also find $f^{(n)} \sim {(f^\dagger)}^n f$.

  \eqref{iter-der-dagger-small} Let $r$ be the unique real number such that $f^\dagger - \frac{r}{T} \prec \frac{1}{T}$. Then $f'T - rf \prec f$. Since $f \not\asymp 1$, we find
  \[ f''T + f' - rf' \prec f', \quad \text{or in other words} \quad {(f')}^\dagger - \frac{r - 1}{T} \prec \frac{1}{T}. \]
  In particular, ${(f')}^\dagger \preceq \frac{1}{T}$. Since $f \not\asymp T^k$ for all $k$, we have $f^{(k)} \not\asymp 1$ for all $k$, so we deduce by induction that
  \[ {\left(f^{(n)}\right)}^\dagger - \frac{r - n}{T} \prec \frac{1}{T}. \]
  In turn, ${(f^{(n)})}^\dagger \asymp \frac{1}{T}$ for $n \neq r$, and ${(f^{(n)})}^\dagger \prec \frac{1}{T}$ if $n = r$, which can only happen if $r \in \Nb$. Let $\ell = r + 1$ if $r \in \Nb$ and $\ell = 0$ otherwise to recover the desired conclusion. For $n \in \Nb$ we have $f^{(\ell+n+1)} = {(f^{(\ell+n)})}^\dagger f^{(\ell + n)} \asymp \frac{f^{(\ell + n)}}{T}$, so by induction $f^{(\ell+n)} \asymp \frac{f^{(\ell)}}{T^n}$. Similarly, $f^{(n+1)} = {(f^{(n)})}^\dagger f^{(n)} \preceq \frac{f^{(n)}}{T}$, hence $f^{(n)} \preceq \frac{f}{T^n}$.
\end{proof}

Two illustrative examples are the following. Take $f = e^{e^T}$. In this case, $f^\dagger = e^T \succ \frac{1}{T}$, and \Cref{iter-der-dagger} predicts that $f^{(n)} \sim {(f^\dagger)}^n f$. Indeed,
\[ f' = e^T e^{e^T} = f^\dagger f, \quad f'' = e^T e^{e^T} + e^{2T}e^{e^T} \sim {\left(f^\dagger\right)}^2 f, \quad f''' \sim e^{3T}e^{e^T} = {\left(f^\dagger\right)}^3 f \quad \dotsc \]
Now take $f = T\log(T)$. Here $f^\dagger = \frac{\log(T) + 1}{T\log(T)} \sim \frac{1}{T}$, so we expect to see $f^{(n+1)} \asymp \frac{f^{(n)}}{T}$, with at most one exceptional $f^{(n+1)} \prec \frac{f^{(n)}}{T}$. Indeed,
\[ f' = \log(T) + 1 \sim \frac{f}{T}, \quad f'' = \frac{1}{T} \prec \frac{f'}{T}, \quad f''' = - \frac{1}{T^2} \asymp \frac{f''}{T}, \quad f'''' = \frac{2}{T^3} \asymp \frac{f'''}{T} \quad \dotsc \]

\begin{cor}\label{radius-of-convergence}
  Let $f \in \OM$ be non-zero such that $f \not\asymp T^k$ for all $k \in \Nb$. Suppose that $\delta \neq 0$. Then the sequence $(f^{(n)} \circ x)\delta^{n}$ is strictly $\prec$-decreasing when $\delta \prec x$, and $(f^\dagger \circ x)\delta \prec 1$, weakly $\prec$-increasing when $\delta \preceq x$ and $(f^\dagger \circ x)\delta \preceq 1$, and eventually strictly  $\prec$-increasing otherwise.
\end{cor}
\begin{proof}
  We apply \Cref{iter-der-dagger} to the ratio between two successive elements of the sequence:
  \[ \frac{(f^{(n+1)} \circ x)\delta^{n+1}}{(f^{(n)} \circ x)\delta^{n}} = \left(\left(\frac{f^{(n+1)}}{f^{(n)}}\right)\circ x\right) \delta =  \left({\left(f^{(n)}\right)}^\dagger \circ x\right)\delta. \]

  When $f^\dagger \succ \frac{1}{T}$, then ${(f^{(n)})}^\dagger \sim f^\dagger$ (\Cref*{iter-der-dagger}\eqref{iter-der-dagger-large}), so
  \[ \left({\left(f^{(n)}\right)}^\dagger \circ x\right)\delta \sim \left(f^\dagger \circ x\right)\delta. \]
  The sequence is then strictly $\prec$-decreasing when $(f^\dagger \circ x)\delta \prec 1$, which implies $\delta \prec x$ since $f^\dagger \circ x \succ \frac{1}{x}$, weakly $\prec$-increasing when $(f^\dagger \circ x)\delta \preceq 1$, which also implies $\delta \prec x$, and it is strictly $\prec$-increasing otherwise.

  When $f^\dagger \preceq \frac{1}{T}$, then ${(f^{(n)})}^\dagger \preceq \frac{1}{T}$, with ${(f^{(n)})}^\dagger \asymp \frac{1}{T}$ for all but possibly one value of $n$ (\Cref*{iter-der-dagger}\eqref{iter-der-dagger-small}), hence
  \[ \left({\left(f^{(n)}\right)}^\dagger \circ x\right)\delta \preceq \frac{\delta}{x}, \]
  with equivalence for all but possibly one $n$. The sequence is then strictly $\prec$-decreasing when $\delta \prec x$, which implies $(f^\dagger \circ x)\delta \prec 1$ since $f^\dagger \circ x \preceq \frac{1}{x}$, weakly $\prec$-decreasing when $\delta \preceq x$, which implies $(f^\dagger \circ x)\delta \preceq 1$, and it is eventually strictly $\prec$-increasing otherwise.
\end{proof}

\begin{cor}\label{iter-taylor-first-term}
  Let $f \in \OM$ be non-zero such that $f \not\asymp T^k$ for all $k \in \Nb$. If $x + \delta \asymp x$ and $(f^\dagger \circ x)\delta \preceq 1$, then $f^{(n)} \circ (x + \delta) \asymp f^{(n)} \circ x$ for all $n \in \Nb$.
\end{cor}
\begin{proof}
  Note that $\delta \preceq x$. By \Cref{radius-of-convergence}, for every $n \in \Nb$ we have
  \[ {\left({\left(f^{(n)}\right)}^\dagger \circ x\right)}\delta = \frac{\left(f^{(n+1)} \circ x\right)\delta^{n+1}}{\left(f^{(n)} \circ x\right)\delta^n} \preceq 1, \]
  thus the conclusion follows from \Cref*{taylor-first-term} applied to each $f^{(n)}$.
\end{proof}

\begin{proof}[Proof of \Cref{taylor}]
  When $\ER(f) = 0$, $f = \log^{\circ k}(T)$, and the conclusion follows directly from the Taylor expansion of $\log$. For $\ER(f) > 0$, fix some $x$, $\delta$ as in the assumptions. Note that $\delta \preceq x$. Write $f = \sum_\gamma r_\gamma e^\gamma$ for $\gamma$ ranging in $\Jb$, where by construction $r_\gamma \neq 0$ implies $\ER(\gamma) < \ER(f)$. Split $f$ as follows:
  \[ f_0 = \sum_{(\gamma' \circ x)\delta \preceq 1} r_\gamma e^\gamma, \quad f_1 = f - f_0 = \sum_{(\gamma' \circ x)\delta \succ 1} r_\gamma e^\gamma. \]

  We first show that $f_1$ can be ignored. Suppose that $f_1 \neq 0$. On the one hand, by construction $(f_1^\dagger \circ x)\delta \sim (\gamma' \circ x)\delta \succ 1$, where $e^\gamma$ is the leading monomial of $f_1$, hence $f_1^\dagger \succ f^\dagger$, thus necessarily $f_1 \prec f$ and $f_1 \prec 1$, whence $f_1 \not\asymp T^k$ for all $k \in \Nb$; since $f \not\asymp T^k$ for all $k$, $f^{(n)} \succ f_1^{(n)}$. On the other, $(f_1^\dagger \circ x) \succ \frac{1}{\delta} \succeq \frac{1}{x}$, so we cannot have $f_1^\dagger \preceq \frac{1}{T}$, hence $f_1^\dagger \succ \frac{1}{T}$, and so $f_1^{(n)} \sim {(f_1^\dagger)}^n f_1$ by \Cref*{iter-der-dagger}\eqref{iter-der-dagger-large}. It follows that
  \[ f_1 \circ x \sim \frac{f_1^{(n)} \circ x}{{(f_1^\dagger \circ x)}^n} \prec \frac{f^{(n)} \circ x}{{(f_1^\dagger \circ x)}^n} \prec (f^{(n)} \circ x)\delta^n. \]

  Similarly, consider $f_1 \circ (x + \delta)$. Note that $(f_1^\dagger \circ (x + \delta))\delta \succ 1$: if not, by \Cref{iter-taylor-first-term} applied to $f_1$ at $x + \delta$, with $-\delta$ in place of $\delta$, we would get
  \[ 1 \prec (f_1^\dagger \circ x)\delta = \frac{f_1' \circ x}{f_1 \circ x}\delta \asymp \frac{f_1' \circ (x + \delta)}{f_1 \circ (x + \delta)}\delta = (f_1^\dagger \circ (x + \delta))\delta \preceq 1, \]
  a contradiction. Therefore, just as in the previous argument, we find
  \[ f_1 \circ (x + \delta) \sim \frac{f^{(n)} \circ (x + \delta)}{{(f_1^\dagger \circ (x + \delta))}^n} \prec (f^{(n)} \circ (x + \delta))\delta^n. \]
  Since $f^{(n)} \circ (x + \delta) \asymp f^{(n)} \circ x$ by \Cref{iter-taylor-first-term},
  \[ f_1 \circ (x + \delta) \prec (f^{(n)} \circ x)\delta^n. \]
  Therefore,
  \[ f \circ (x + \delta) - f \circ x = f_0 \circ (x + \delta) - f_0 \circ x + o\left( (f^{(n)} \circ x)\delta^n \right)\]

  Since $f \asymp f_0$, so $f^{(n)} \asymp f_0^{(n)}$, it is now enough to prove the conclusion with $f_0$ in place of $f$. Suppose that $e^\gamma$ is a monomial in the support of $f_0$, thus $\gamma \in \Jb$, $(\gamma' \circ x)\delta \preceq 1$, and $\ER(\gamma) < \ER(f_0) \leq \ER(f)$. We distinguish two cases.

  If $\gamma \asymp T^k$ for some $k \in \Nb$, then $k > 0$, hence $\gamma^\dagger \asymp \frac{1}{T}$. For each monomial $\mf$ in the support of $\gamma$, we have $\gamma \succeq \mf \succ 1$, so in particular $\mf^\dagger \preceq \gamma^\dagger \asymp \frac{1}{T}$, hence $(\mf^\dagger \circ x)\delta \preceq \frac{\delta}{x} \preceq 1$. When $\mf \not\asymp T^d$ for all $d \in \Nb$, since $\ER(\mf) \leq \ER(\gamma) < \ER(f)$, we may apply the inductive hypothesis to deduce
  \[ \mf \circ (x + \delta) = \sum_{i = 0}^{n-1} \frac{\mf^{(i)} \circ x}{i!}\delta^i + O\left(\mf^{(n)} \delta^n\right). \]
  When $\mf \asymp T^d$ for some $d \in \Nb$, then in fact $\mf = T^d$, hence the above approximation still holds by the binomial theorem and the fact that $\delta \preceq x$. Moreover, in either case $\mf^{(i)} \preceq \frac{\mf}{T^{i}} \preceq T^{k - i}$ (by \Cref*{iter-der-dagger}\eqref{iter-der-dagger-small} in the former case, trivially in the latter), so $(\mf^{(i)} \circ x)\delta^i \preceq 1$ for all $i > 0$. By strong linearity of composition and derivation, we can sum all the terms of $\gamma$ to deduce that $(\gamma^{(i)} \circ x)\delta^i \preceq 1$ for all $i > 0$, and that
  \[ \gamma \circ (x + \delta) = \sum_{i = 0}^{n - 1}\frac{\gamma^{(i)} \circ x}{i!}\delta^i + O\left(\gamma^{(n)} \delta^n\right). \]

  If $\gamma \not\asymp T^k$ for all $k \in \Nb$, then we simply observe that $(\gamma^\dagger \circ x)\delta \prec 1$, since $\gamma \succ 1$, so the above equality holds by inductive hypothesis, and $(\gamma^{(i)} \circ x)\delta^i \preceq 1$ for $i > 0$ by \Cref{radius-of-convergence}.

  Therefore,
  \[ \begin{split}
    f_0 \circ (x + \delta) &= \sum_{(\gamma' \circ x)\delta \preceq 1} r_\gamma e^{\gamma \circ (x + \delta)} = \sum_{(\gamma' \circ x)\delta \preceq 1} r_\gamma e^{\sum_{i = 0}^{n - 1} \frac{\gamma^{(i)} \circ x}{i!}\delta^i + O\left((\gamma^{(n)} \circ x)\delta^n\right)} \\
    & = \sum_{(\gamma' \circ x)\delta \preceq 1} r_\gamma e^{\gamma \circ x} \exp\left(\sum_{i = 1}^{n - 1} \frac{\gamma^{(i)} \circ x}{i!}\delta^i + O\left((\gamma^{(n)} \circ x)\delta^n\right)\right) \\
    & = \sum_{(\gamma' \circ x)\delta \preceq 1} \left(\sum_{i = 0}^{n - 1} \frac{{(r_\gamma e^\gamma)}^{(i)} \circ x}{i!}\delta^i + O\left(({(r_\gamma e^\gamma)}^{(n)} \circ x)\delta^n\right)\right) \\
    & = \sum_{i = 0}^{n - 1} \frac{f_0^{(i)} \circ x}{i!}\delta^i + O\left((f_0^{(n)} \circ x)\delta^n\right),
  \end{split} \]
  where on the second line, the argument of $\exp$ is $\preceq 1$, since $(\gamma^{(i)} \circ x)\delta^i \preceq 1$ for $i > 0$, and so we may use the fact that $e^y = 1 + y + \dotsb + \frac{y^{n-1}}{(n-1)!} + O(y^n)$ for any $y \preceq 1$ to proceed to the following step.
\end{proof}

\begin{rem}\label{validity-at-boundary}
  The conclusion of \Cref{taylor} loses its significance at the boundary: if $f^\dagger \succ \frac{1}{T}$ and $(f^\dagger \circ x)\delta \asymp 1$, then the error terms all have the same size by \Cref{iter-der-dagger}, and the conclusion collapses to $f \circ (x + \delta) = O(f \circ x)$; a comparable remark can be made for $f^\dagger \preceq \frac{1}{T}$ and $\delta \succeq x$, where the error terms can get smaller at most once.

  Error terms even increase in size if $(f^\dagger \circ x)\delta \succ 1$ or $\delta \succ x$. In those cases, the conclusion of \Cref{taylor} may or may not be valid depending on $f$. Consider the `first-order approximation'
  \[ f \circ (x + \delta) - f \circ x = O((f' \circ x)\delta). \]

  When $(f^\dagger \circ x)\delta \succ 1$, consider $f = e^T$, thus assume $\delta \succ 1$. Then the first-order approximation is valid for $f = e^T$ if and only if $\delta < 0$. More generally, the approximation remains valid if $f \succ 1$ and $\delta < 0$, or if $f \prec 1$ and $\delta > 0$, since \Cref{monotonicity} implies $f \circ (x + \delta) \preceq f \circ x$.

  For $\delta \succ x$, the first-order approximation is valid for $f = \log(T)$ and it fails for $f = \sqrt{T^3}$. Analogy with real functions suggests that the approximation is valid for $\delta \succ x$ exactly when $f \preceq T$. This is related to whether $f'' \geq 0$ implies that $f$ is convex, namely $f \circ (x + \delta) - f \circ x \geq (f' \circ x)\delta$ for $\delta \geq 0$. As alluded to in the introduction, this does not seem to follow in a direct way from \Cref{monotonicity}.

  The boundary $x + \delta \prec x$ is more subtle: the error terms do not increase in size, but the Taylor approximation may still hold or fail. For example, the first-order approximation fails for $f = \frac{1}{T}$ and $f = \log(T)$ (note that $\delta \asymp x$, so $(f^\dagger \circ x)\delta \preceq 1$ in both examples):
  \[ \frac{1}{x + \delta} - \frac{1}{x} = -\frac{\delta}{(x + \delta)x} \succ - \frac{\delta}{x^2}, \quad \log(x + \delta) - \log(x) = \log\left(\frac{x + \delta}{x}\right) \succ 1 \asymp \frac{\delta}{x}. \]
  On the other hand, the approximation is valid for $f = \sqrt{T}$:
  \[ \sqrt{x + \delta} - \sqrt{x} = \frac{\delta}{\sqrt{x + \delta} + \sqrt{x}} \asymp -\frac{\delta}{2\sqrt{x}}. \]
  We can in fact give a full classification. Assume $x + \delta \prec x$ and $f^\dagger \preceq \frac{1}{T}$ (we ignore $f^\dagger \succ \frac{1}{T}$, as in that case $(f^\dagger \circ x)\delta \preceq 1$ implies $\delta \prec x$). Then the first-order approximation is valid if and only if $f^\dagger \asymp \frac{1}{T}$ and $f \succ 1$.

  Indeed, suppose that $f^\dagger \asymp \frac{1}{T}$. Note that the first-order approximation collapses to $f \circ (x + \delta) = O(f \circ x)$. We have $\log(f) \sim r\log(T)$ for some non-zero $r \in \Rb$, so
  \[ \log(f \circ (x + \delta)) - \log(f \circ x) \sim r(\log(x + \delta) - \log(x)) \succ 1 \]
  by \Cref{leading-gap}. Note that $\log(x + \delta) - \log(x)$ is negative infinite. When $f \prec 1$, we have $r < 0$, so $f \circ (x + \delta) \succ f \circ x$, hence the approximation fails. When $f \succ 1$, then $r > 0$, so $f \circ (x + \delta) \prec f \circ x$, hence the approximation is valid.

  Now suppose that $f^\dagger \prec \frac{1}{T}$. Let $\varepsilon = \frac{f \circ (x + \delta) - f \circ x}{f \circ x}$. We claim that $f$ satisfies the first-order approximation if and only if $\log|f|$ does. Since $(f' \circ x)\delta \asymp (f' \circ x)x \prec f \circ x$, the approximation for $f$ implies $\varepsilon \prec 1$, while the one for $\log|f|$ implies $\log(1 + \varepsilon) \prec 1$. Crucially, in either case $\varepsilon \sim \log(1 + \varepsilon)$. The claim follows at once from $(\log|f|)' = \frac{f'}{f}$. Moreover, ${(\log|f|)}^\dagger = \frac{f^\dagger}{\log|f|} \prec f^\dagger \prec \frac{1}{T}$, since $f \not\asymp 1$. Therefore, to check whether $f$ satisfies the first-order approximation, we may replace $f$ with $\log|f|$ until $f \sim \log^{\circ k}(T)$. Since $f' \sim (\log^{\circ k}(T))'$, \Cref{leading-gap} implies that we may further replace $f$ with $\log^{\circ k}(T)$. Applying the same argument in reverse, we may replace $\log^{\circ k}(T)$ with $\log(T)$. The approximation fails for $\log(T)$, hence it fails for the starting $f$.
\end{rem}

\begin{cor}\label{taylor-general}
  Let $f \in \OM$ and $x, \delta \in \Ub$ with $x > \Rb$, $x + \delta > \Rb$, $x + \delta \asymp x$. If $f \asymp T^k$ for some $k \in \Nb$ and $f^{(k+1)} \neq 0$, suppose that $({(f^{(k+1)})}^\dagger \circ x)\delta \preceq 1$; otherwise, if $f \neq 0$, suppose that $(f^\dagger \circ x)\delta \preceq 1$. Then for all $n \geq 0$,
  \[ f \circ (x + \delta) = \sum_{i = 0}^{n-1} \frac{f^{(i)} \circ x}{i!}\delta^i + O\left((f^{(n)} \circ x)\delta^n\right). \]
\end{cor}
\begin{proof}
  If $f \not\asymp T^k$ for all $k \in \Nb$ and $f \neq 0$, this is just \Cref{taylor}. If $f = 0$, the conclusion is trivial.

  Now suppose that $f \asymp T^k$ for some $k \in \Nb$. Let $p$ be the sum of all the terms of $f$ that are $\asymp T^d$ for some $d \in \Nb$, where necessarily $d \leq k$. Note that $p \asymp f \asymp T^k$. Then $p$ is a polynomial in $T$ of degree $k$, and by construction $g = f - p$ satisfies $g \not\asymp T^d$ for all $d \in \Nb$. In particular, $p^{(k+1)} = 0$, so $f^{(k+1)} = g^{(k+1)}$.

  The conclusion of \Cref{taylor} is valid for $p$ by the binomial theorem and the fact that $\delta \preceq x$ (and for $n > k$, it is true even for $\delta \succ x$, as the error term becomes $0$). If $g = 0$, we are done. If $g \neq 0$, we distinguish two cases. If $g^\dagger \preceq \frac{1}{T}$, then $\delta \preceq x$ implies $(g^\dagger \circ x)\delta \preceq \frac{\delta}{x} \preceq 1$. If $g^{\dagger} \succ \frac{1}{T}$, then ${(f^{(k+1)})}^\dagger = {(g^{(k+1)})}^\dagger \sim g^\dagger$ by \Cref*{iter-der-dagger}\eqref{iter-der-dagger-large}, so the assumptions guarantee that $(g^\dagger \circ x)\delta \preceq 1$. In either case, we can apply \Cref{taylor} to $g$. The conclusion now follows immediately from $f = p + g$ and the observations $f^{(n)} \asymp p^{(n)} \succ g^{(n)}$ for $n \leq k$, $f^{(n)} = g^{(n)}$ for $n > k$.
\end{proof}

\bibliographystyle{alphaurl}
\bibliography{mscrefs}

\end{document}